\documentclass[reqno,a4paper,12pt]{amsart}
\usepackage{amsmath,amssymb,amsthm,amscd,amsfonts,amsbsy,bm}
\usepackage[english]{babel}
\usepackage[mathscr]{eucal}

\font\sc=rsfs10 at 12pt

\numberwithin{equation}{section}

\renewcommand{\a}{\alpha}
\renewcommand{\b}{\beta}
\newcommand{\g}{\gamma}

\renewcommand{\d}{\delta}
\newcommand{\D}{\Delta}
\newcommand{\e}{\epsilon}

\newcommand{\z}{\zeta}

\renewcommand{\k}{\kappa}

\renewcommand{\l}{\lambda}
\renewcommand{\L}{\Lambda}
\newcommand{\m}{\mu}

\newcommand{\x}{\xi}

\renewcommand{\r}{\rho}
\newcommand{\s}{\sigma}
\newcommand{\Ss}{\Sigma}

\newcommand{\f}{\phi}
\newcommand{\ff}{\varphi}
\newcommand{\F}{\Phi}

\newcommand{\p}{\psi}
\renewcommand{\P}{\Psi}
\renewcommand{\o}{\omega}
\renewcommand{\O}{\Omega}
\newcommand{\vs}{\varsigma}

\newcommand{\C}{{\mathbb C}}
\newcommand{\R}{{\mathbb R}}
\newcommand{\Z}{{\mathbb Z}}

\newcommand{\K}{{\mathbb K}}
\newcommand{\ab}{{\mathbf a}}

\newcommand{\eb}{{\mathbf e}}

\newcommand{\kb}{{\mathbf k}}
\newcommand{\lb}{{\mathbf l}}
\newcommand{\mb}{{\mathbf m}}
\newcommand{\nb}{{\mathbf n}}

\newcommand{\pb}{{\mathbf p}}
\newcommand{\qb}{{\mathbf q}}
\newcommand{\rb}{{\mathbf r}}

\newcommand{\wb}{{\mathbf w}}
\newcommand{\xb}{{\mathbf x}}

\newcommand{\zb}{{\mathbf z}}

\newcommand{\Kb}{{\mathbf K}}

\newcommand{\Pb}{{\mathbf P}}

\newcommand{\Tb}{{\mathbf T}}

\newcommand{\SF}{\mathfrak S}

\newcommand{\Ac}{{\mathcal A}}
\newcommand{\Bc}{{\mathcal B}}
\newcommand{\Cc}{{\mathcal C}}
\newcommand{\Dc}{{\mathcal D}}
\newcommand{\Ec}{{\mathcal E}}
\newcommand{\Fc}{{\mathcal F}}

\newcommand{\Hc}{{\mathcal H}}

\newcommand{\Jc}{{\mathcal J}}

\newcommand{\Lc}{{\mathcal L}}

\newcommand{\Pc}{{\mathcal P}}

\newcommand{\Rc}{{\mathcal R}}
\newcommand{\Sc}{{\mathcal S}}
\newcommand{\Tc}{{\mathcal T}}
\newcommand{\Zc}{{\mathcal Z}}

\newcommand{\Ls}{\sc\mbox{L}\hspace{1.0pt}}
\newcommand{\Ms}{\sc\mbox{M}\hspace{1.0pt}}
\newcommand{\As}{\sc\mbox{A}\hspace{1.0pt}}
\newcommand{\Hs}{\sc\mbox{H}\hspace{1.0pt}}

\newcommand{\Zs}{\sc\mbox{Z}\hspace{1.0pt}}
\newcommand{\Fs}{\sc\mbox{F}\hspace{1.0pt}}

\newcommand{\Ns}{\sc\mbox{N}\hspace{1.0pt}}

\newcommand{\Span}{{\rm Span}\,}

\newcommand{\supp}{\hbox{{\rm supp}}\,}

\newcommand{\Det}{{\rm Det}\,}
\newcommand{\lp}{\left(}
\newcommand{\rp}{\right)}
\newcommand{\blp}{\bigl(}
\newcommand{\brp}{\bigr)}

\newcommand{\card}{\operatorname{card}}

\DeclareMathOperator{\im}{{\rm Im}\,} \DeclareMathOperator{\re}{{\rm Re}\,}
\DeclareMathOperator{\rank}{rank}

\newcommand{\Ker}{\hbox{{\rm Ker}}\,}

\newcommand{\psupp}{\operatorname{psupp}}
\newcommand{\singsupp}{\operatorname{sing supp}}

\newtheorem{theorem}{Theorem}[section]
\newtheorem{proposition}[theorem]{Proposition}
\newtheorem{lemma}[theorem]{Lemma}

\theoremstyle{definition}
\newtheorem{definition}[theorem]{Definition}

\theoremstyle{remark}
\newtheorem{remark}[theorem]{Remark}
\newtheorem{example}[theorem]{Example}

\begin{document}

\title[Finite rank Toeplitz operators]{Finite rank Toeplitz operators in  Bergman spaces}
 %\dedicatory{}

\author[G. Rozenblum]{Grigori Rozenblum}
\address{1) Department of Mathematics
                        Chalmers University of Technology,  2) Department of Mathematics
                         University of Gothenburg }
\email{grigori@math.chalmers.se}

\begin{abstract}
We discuss resent developments in the problem of description of finite rank Toeplitz operators in different Bergman spaces and give some applications
\end{abstract}
%\end{keywords}
\date{}

\maketitle

 %\tableofcontents
%\begin{dedication}To Volodya Maz'ya, an outstanding mathematician and person\end{dedication}
\section{Introduction}
\label{Section.intro}

Toeplitz operators arise in many fields of Analysis and have been an object of active study for many years. Quite a lot of questions can be asked about these operators, and these questions depend on the field where Toeplitz operators are applied.

The classical Toeplitz operator $T_f$  in the Hardy space $H^2(S^1)$ is defined as
\begin{equation}\label{1.T}
    T_f u= Pfu,
\end{equation}
for $u\in H^2(S^1)$, where $f$ is a bounded function on $S^1$ (the weight function) and $P$ is the Riesz projection, the orthogonal projection $P: L_2(S^1)\to H^2(S^1)$. Such operators are often called Riesz-Toeplitz or Hardy-Toeplitz operators (see, \cite{MartRosen}, for more detail). More generally, for a Hilbert space $\Hc$ of functions and a closed subspace $\Lc\subset\Hc$, the  Toeplitz operator  $T_f$ in $\Lc$ acts as in \eqref{1.T}, where $P$ is the projection $P:\Hc\to \Lc$. In particular, in the case when $\Hc$ is the space $L_2(\Omega,\r)$ for some domain $\O\subset \C^d$ and some measure $\r$ and $\Lc$ is the Bergman space  $\Bc^2=\Bc^2(\O,\r)$ of analytical functions in $\Hc$, such operator is called Bergman-Toeplitz; we will denote it by $\Tc_f$.

Among many interesting properties of Riesz-Toeplitz operators,  we mention the following \emph{cut-off} one. If ${f}$ is a bounded function and the operator $T_f$ is compact then $f$ should be zero. For many other classes of operators a similar cut-off on some level is also observed.  The natural question arises, whether there is a kind of cut-off property for Bergman-Toeplitz operators. Quite long ago it became a common knowledge that at least direct analogy does not take place. In the paper \cite{Lue}, the conditions were found on the function $f$ in the unit disk $\O=D$ guaranteeing that the operator $\Tc_f$ in $\Bc^2(D,\l)$ with Lebesgue measure $\l$ belongs to the Schatten class $\SF_p$.  So, the natural question came up: probably, it is on the finite rank level that the cut-off takes place. In other words, if a Bergman-Toeplitz operator has finite rank it should  be zero.

It was known long ago that the Schatten class behavior of $\Tc_f$ is determined by the rate of convergence to zero at the boundary of the function $f$. Therefore the finite rank (FR) hypothesis deals with functions $f$ with compact support not touching the boundary of $\O$. In this setting the FR hypothesis is equivalent to the one for Toeplitz operators on the Bargmann (Fock, Segal) space consisting of analytical functions in $\C$, square summable with a Gaussian weight.
A proof of the FR hypothesis appeared in the same paper \cite{Lue}, about twenty lines long. Unfortunately, there was an unrepairable fault in the proof, so the FR remained unsettled.

It was only in 2007 that the proof of the FR hypothesis was finally found, even in a more general form. The Bergman projection $\Pb:L_2\to\Bc$ can be extended to an operator from the space of distributions $\Dc'(\O)$ to $\Bc^2(\O,\l)$. Let $\m$ be a regular complex Borel measure with compact support in $\O$. With $\m$ we associate the Toeplitz operator    $\Tc_\m:\ u\mapsto \Pb u\mu$ in $\Bc^2(\O,\l)$.

 In the paper \cite{Lue2} the following result was established.
\begin{theorem}\label{1.Lue}Suppose that the Toeplitz operator $\Tc_\m$ in $\Bc^2(\O,\l)$, $\O\subset \C$ has finite rank $\rb$. Then the measure $\m$ is the sum of $\rb$ point masses,
\begin{equation}\label{1.mu}
    \m=\sum_1^\rb C_k \d_{z_j}, \ z_j\in \O.
\end{equation}
\end{theorem}

 The publication of the proof of Theorem \ref{1.Lue} induced an activity around it.  In two  years to follow several papers appeared, where the FR theorem was generalized in different directions, and interesting applications were found in Analysis and Mathematical Physics.

 In this paper we aim for  collecting and systematizing the existing results on the finite rank problem and their applications. We also present several new theorems generalizing and extending  these results.

 \section{Problem setting}\label{setting}

 Let $\Omega$
 be a domain in $\R^d$ or $\C^d$. We suppose that a measure $\r$ is defined on $\O$, jointly absolutely continuous with  Lebesgue measure. Suppose that $\Lc$ is a closed subspace in $\Hc=L_2(\Omega, \r)$, consisting of smooth functions, $\Lc\subset C^{\infty}(\Omega)$. In this case the orthogonal projection $\Pb:\Lc\to \Hc$ is an integral operator with smooth kernel,
 \begin{equation}\label{2.Bergman}
    \Pb u(x)=\int P(x,y)u(y)d\r(x).
 \end{equation}
 We will call $\Pb$ the Bergman projection and $P(x,y)$ the Bergman kernel (corresponding to the subspace $\Lc$).

 Let $F$ be a distribution, compactly supported in $\O,$ $F\in \Ec'(\O)$. We will denote by $\langle F,\f\rangle$ the action of the distribution $F$ on the function $\f\in \Ec$. Then one can define the Toeplitz operator in $\Lc$ with  \emph{weight} $F$:

 \begin{equation}\label{2.ToeplitsDef1}
    (\Tc_F u)(x) =\langle F,P(x,\cdot)u(\cdot)\rangle.
 \end{equation}
 The formula \eqref{2.ToeplitsDef1} can be also understood in the following way. The operator $\Pb$ considered as  an operator $P:\Hc\to \Lc$ has an adjoint, $P':\Lc'\to\Hc$, so   $PP'$ is the extension of $
 \Pb$ to the operator $\Lc'\to\Lc$, in particular, $\Pb$ extends as an operator from $\Ec'(\O)$  to $ \Lc$. In this setting, $Fu \in \Ec'(\O)$ for $u\in \Ec(\O)$ and the Toeplitz operator has the form
 \begin{equation}\label{2.ToeplitzDef2}
    \Tc_F u = \Pb Fu,
 \end{equation}
 consistently with the traditional definition of Toeplitz operators.

 It is more convenient to use the description of the Toeplitz operator by means of the sesquilinear form. For $u, v\in \Lc$, we have
 \begin{equation}\label{2.ObtainingQF}
    (\Tc_F u, v)=(\Pb  Fu, v)=\langle \s Fu,\overline{\Pb v}\rangle=\langle \s F,  u\bar{v} \rangle,
 \end{equation}
 where $\s$ is the Radon-Nicodim derivative of $\r$ with respect to the Lebesgue measure.
 In particular, if $F$ is a regular Borel complex measure $F=\m$,
 the corresponding Toeplitz operator acts as
 \begin{equation}\label{2.ToeplMeasure}
    \Tc_\m u(x)=\int_{\O} P(x,y)u(x)d\mu(x),
 \end{equation}
 and the quadratic form is
 \begin{equation*}
    (\Tc_F u, v)=\int_{\O} u\bar{v}\s d\m(x).
 \end{equation*}
 Finally, when $F$ is a bounded function,
 the formula \eqref{2.ObtainingQF} takes the form
 \begin{equation}\label{2.QF}
    (\Tc_Fu,v)=\int_{\O} u\bar{v} F(x) d\r(x).
 \end{equation}

 Classical examples of Bergman spaces and corresponding Toeplitz operators are produced by solutions of elliptic equations and systems.
 \begin{example}\label{2.E1} Let $\O$ be a bounded domain in $\C$, $\r=\l$ be the Lebesgue measure, $\Lc=\Bc^2(\O)$ be the space of $L_2$ functions analytical in $\O$. This is the classical Bergman space.\end{example}
 \begin{example}\label{2.E2} Let $\O$ be a bounded pseudoconvex domain in $\C^d$, $d>1$, with Lebesgue measure $\r$ and let the space $\Lc$ consist of $L_2$ functions analytical in $\O$. This is also a classical Bergman space. Here, and in Example \ref{2.E1}, measures different from the Lebesgue one are also considered, especially when $\O$ is a ball or a (poly)disk.\end{example}
 \begin{example}\label{2.E3} For a bounded domain $\O\subset R^d$, we set $\Lc$ to be the space of $L_2$ solutions of the  equation $Lu=0$, where $L$ is an elliptic differential  operator with constant coefficients. In particular, if $L$ is the Laplacian, the space $\Lc$ is called the harmonic Bergman space.\end{example}
 \begin{example}\label{2.E4} If $\O$ is a bounded domain in $\R^d$ with \emph{even} $d=2\mb$, and $\R^d$ is identified with $\C^\mb$ with variables $z_j=(x_j,y_j), j=1,\dots,\mb$, the Bergman space of functions which are harmonic with respect to each pair $(x_j,y_j)$ is called $\mb$-harmonic Bergman space; if on the other hand, the space of functions $u(z)$ such that $u_\z(\x_1,\x_2)=u(\z(\x_1+i\x_2))$ is harmonic as a function of variables $\x_1,\x_2$ for any $\z\in \C^m\setminus\{0\}$,  is called pluriharmonic Bergman space.\end{example}
 \begin{example}\label{2.E5}Let $\O$ be the whole of $\C^\mb=\R^d$, with the Gaussian measure $d\r=\exp(-|z|^2/2)d\l$. The subspace $\Lc\subset L_2(\C^\mb,\rho) $ of entire analytical functions in $\C^\mb$ is called \emph{Fock} or \emph{Segal-Bargmann} space.
 \end{example}

 The study of Toeplitz operators in many cases is based upon the consideration of associated infinite matrices.

 Let $\Fs_1=\{f_j(x), x\in \O\}$, $\Fs_2=\{g_j(x), x\in \O\}$ be two infinite systems of functions in $\Lc$. With these systems and a distribution $F\in\Ec'(\O)$ we associate the matrix
 \begin{equation}\label{2:matrix}
    \Ac=\Ac(F)=\Ac(F,\Fs_1,\Fs_2,\O,\rho)=(\Tc_F f_j,g_k)_{j,k=1,\dots}=(\langle \s F,  f_j\bar{g_k}\rangle).
 \end{equation}
 So, the matrix $\Ac$ is the matrix of the sesquilinear form of the operator $\Tc_F$ on the systems $\Fs_1, \Fs_2$. We formulate the obvious but important statement.
 \begin{proposition}\label{2.Prop.Matrix} Suppose that the Toeplitz operator $\Tc_F$ has finite rank $\rb$. Then the matrix $\Ac$ also has finite rank, moreover $\rank(\Ac)\le \rb.$\end{proposition}

 The use of matrices of the form \eqref{2:matrix} enables one to perform important reductions. In particular, since the domain $\O$ does not enter explicitly into the matrix, the rank of this  matrix does not depend on the domain $\O$, as long as one can chose the systems $\Fs_1, \Fs_2$ dense simultaneously in the  Bergman spaces  in different domains.      Thus,   in  particular, the FR  problems      for   the analytical  Bergman spaces in bounded domains and for the Fock  space   are    equivalent  (see    the discussion in  \cite{RShir}.)

 \section{Theorem of D. Luecking. Extensions in dimension 1}\label{SectionLuecking}
 In this section we present the original proof given by  D. Luecking in \cite{Lue2}, and give  extensions in several  directions.
 \begin{theorem}\label{3.Thm.Luecking}
 Let $\O\subset\C$ be a bounded domain, with Lebesgue measure. Suppose that for some regular complex Borel measure $\m$, absolutely continuous with respect to the Lebesgue measure, with compact support in $\O$, the Toeplitz operator $\Tc_\m$ in the Bergman space of analytical functions has finite rank $\rb$. Then $\mu=0$.
 \end{theorem}
  We formulate and prove here Luecking's theorem only in    in  the case of an absolutely continuous measure; the case of more singular measures will be taken care of later, as a part of the   general distributional setting.
 In the  proof, which follows \cite{Lue2}, we separate a lemma that will be used  further on.
 \begin{lemma}\label{3.LueLem1} Let $\f$ be a linear functional on polynomials in $z,\bar{z}$. Denote by $\Ac(\f)$  the matrix with elements $\f(z^j\bar{z}^k)$. Then the following are equivalent:
 \begin{enumerate}\item the matrix $\Ac(\f)$ has finite rank not greater than $\rb$;
 \item for any collections of nonnegative integers $J=\{j_0,\dots,j_\rb\}$, $K=\{k_0,\dots,k_\rb\},$
     \begin{equation}\label{3.VanishingMod}
     \textstyle
     \f^{\otimes N}\blp \prod_{i\in(0,\rb)} z_i^{j_i} \det \overline{z_i}^{k_l} \brp = 0,
     \end{equation}
     where $N=\rb+1$.
 \end{enumerate}
 \end{lemma}
 \begin{proof}
 Since passing to linear combinations of rows and columns does not increase the rank of the matrix, it follows that for any polynomials $f_j(z), g_k(z)$, with $j,k=0,\dots,\rb,$ the determinant $\Det(\f(f_j\bar{g}_k))$ vanishes.

The determinant is linear in each column and $\f $ is a linear
functional, so we can write
\begin{equation*}
    \f\left(f_0(z)\times\left|
      \begin{matrix}
        \overline{g_0(z)}\vbox to 14.5pt{} & \mu(f_1\bar g_0) & \dots & \f(f_\rb\bar g_0) \\
        \overline{g_1(z)} & \f(f_1\bar g_1) & \dots & \f(f_\rb\bar g_1) \\
          \vdots          &      \vdots      & \ddots&       \vdots     \\
        \overline{g_\rb(z)} & \f(f_1\bar g_\rb) & \dots & \f(f_\rb\bar g_\rb)
        \vtop to 5pt{}
      \end{matrix}
    \right|\right) = 0
\end{equation*}
We introduce the variable $z_0$ in place of $z$ above and use $\f_0$ for
$\f$ acting in the variable $z_0$. We repeat this process in each
column (using the variable $z_j$ in column $j$ and the notation $\f_j$ for $\f$
acting in $z_j$) to obtain
\begin{equation}\label{vanishing}\textstyle
    \f_0\lp\f_1\lp \dots \f_\rb\lp
        \prod_{k=0}^{\rb}f_k(z_k)\det \blp g_j(z_k)\brp\rp\dots\rp\rp = 0
\end{equation}
We now specialize to the case where each $f_i=z^{j_i}$,  $g_i=z^{k_i}$ and arrive at \eqref{3.VanishingMod}, thus proving the implication $1\Rightarrow2$. The converse implication follows by going along the above reasoning in the opposite direction.
 \end{proof}
\begin{proof}[Proof of Theorem \ref{3.Thm.Luecking}] We identify $\C$ and $\R^2$ with co-ordinates $z=x+iy$. Consider the functional $\f(f)=\f_\m(f)=\int f(z) d\m(z)$.
Write $Z$ for the $N$-tuple $(z_0, z_1,\dots, z_\rb)$ and  $V_J(Z)$
for the determinant $\det\left(z_i^{k_j}\right)$. By Lemma \ref{3.LueLem1},
\begin{equation}\label{3.vanishing}
    \f^{\otimes N}\left(Z^K\overline{V_J(Z)}\right)=0.
\end{equation}

Taking finite sums of
equations \eqref{3.vanishing}, we get for any polynomial $P(Z)$ in $N$
variables:
\begin{equation}\label{vanish2}
    \f^{\otimes N} \left( P(Z)\overline{V_J(Z)} \right) = 0.
\end{equation}
By taking  linear  combinations of antisymmetric polynomials $V_J(Z)$ one can obtain any antisymmetric polynomial $Q(Z)$ (see \cite{Lue2}) for details). Thus
\begin{equation}\label{3.vanish3}
    \f^{\otimes N} \left( P(Z)\overline{Q(Z)} \right) = 0
\end{equation}
for any polynomial $P(Z)$ and any antisymmetric polynomial $Q(Z)$. In its turn, the polynomial $Q(Z)$ is divisible by the lowest degree antisymmetric polynomial, the Vandermonde polynomial $V(Z)=\prod_{0\le j\le k\le \rb}(z_j-z_k)$, $Q(Z)=Q_1(Z)V(Z)$ with a symmetric polynomial $Q_1(Z)$. We write \eqref{3.vanish3} for $Q$ of this form and $P$ having the form $P(Z)=P_1(Z)V(Z)$. So we arrive at
\begin{equation}\label{3.vanish4}
   \f^{\otimes N} \left(P_1(Z)\overline{Q_1(Z)}|V(Z)|^2\right) = 0 \quad\mbox{for all
   symmetric $P_1$ and $Q_1$}.
\end{equation}

It is clear that finite sums of products of the form
$P_1(Z)\overline{Q_1(Z)}$ (with $P_1$ and $Q_1$ symmetric) form an
algebra $\As$ of functions on $\C$ which contains the constants and is
closed under conjugation. It doesn't separate points because each
element is constant on sets of points that are permutations of one
another. Therefore  we  define an equivalence relation $\sim$ on $\C^{N}:$    $Z_1
\sim Z_2$ if and only if $Z_2 = \pi(Z_1)$ for some permutation $\pi$.
Let $Z = (z_0, \dots, z_\rb)$ and $W= (w_0,\dots,w_\rb)$. If $Z\not\sim W$
then the polynomials $p(t) = \prod(t-z_j)$ and $q(t) = \prod(t-w_j)$
have different zeros (or the same zeros with different orders). This
implies that the coefficient of some power of $t$ in $p(t)$ differs from
the corresponding coefficient in $q(t)$. Thus there is an elementary
symmetric function that differs at $Z$ and $W$. Consequently, $\As$
separates equivalence classes.

We  give the quotient space $\C^{N}/{\sim}$ the standard quotient
space topology. If $K$ is any compact set in $\C^{N}$ that is invariant
with respect to $\sim$, then $K/{\sim}$ is compact and Hausdorff. Also,
any symmetric continuous function on $\C^N$ induces a continuous
function on $\C^{N}/{\sim}$ (and conversely). Thus we can apply the
Stone-Weierstrass theorem (on $K/{\sim}$) to conclude that $\As$ is dense
in the space of continuous symmetric functions, in the topology of
uniform convergence on any compact set. Therefore, for any continuous
symmetric function $f(Z)$
\begin{equation}
    \int_{\C^{N}} f(Z) |V(Z)|^2\,d\mu^{\otimes N}(Z) = 0
\end{equation}
If $f$ is an arbitrary continuous function, the above integral will be
the same as the corresponding integral with the symmetrization of $f$ replacing $f$. This is
because the function $|V(Z)|^2$ and the product measure $\mu^{\otimes N}$ are both
invariant under permutations of the coordinates. We conclude that this
integral vanishes for \emph{any continuous $f$} and so the measure
$|V(Z)|^2\,d\mu^{\otimes N}(Z)$ must be zero. Thus, $\mu^{\otimes N}$ is supported on the
set where $V$ vanishes, i.e. on the set of Lebesgue measure zero. Since $\mu^{\otimes N}$ is absolutely continuous, it must be zero.
\end{proof}

 The initial setting of  Theorem \ref{3.Thm.Luecking} dealt with arbitrary measures, as it is explained in the Introduction. A more advanced result was obtained in \cite{AlexRoz}, where Luecking's theorem was carried over to distributions.

 \begin{theorem}\label{3.Th.Alexroz}Suppose that $F\in \Ec'(\O)$ is a distribution with compact support in $\O\subset\C$ and the Toeplitz operator $\Tc_F$ has finite rank $\rb$. Then the distribution $F$ is a finite combination of $\delta$-distributions at some points in  $\O$ and their derivatives,
 \begin{equation}\label{3.AlexRozEq}
    F=\sum_{j\le \rb}L_{j}\d(z-z_j),
 \end{equation}
 $L_j$ being differential operators.
 \end{theorem}

We start with some observations about distributions in $\Ec'(\C)$. For such distribution we denote by
$\psupp F$ the complement of the unbounded component of the complement of $\supp F$.
\begin{lemma}\label{Lem.compsup}Let $F\in \Ec'(\C)$.
Then the following two statements are equivalent:\\
 a) there exists a distribution
$G\in \Ec'(\C)$ such that $\frac{\partial G}{\partial \bar{z}}=F$, moreover $\supp G\subset\psupp
F$;\\ b) $F$ is orthogonal to all polynomials of $z$ variable,  i.e. $\langle F,z^k\rangle=0$ for
all $k\in\Z_+$.
\end{lemma}
 \begin{proof} The implication $a)\Longrightarrow b)$ follows from the relation
\begin{equation}\label{byparts}
    \langle F, z^k\rangle=\langle \frac{\partial G}{\partial \bar{z}},z^k\rangle
    =\langle G ,\frac{\partial z^k}{\partial \bar{z}}\rangle=0.
\end{equation}

We prove that $b)\Longrightarrow a)$.  Put $G:=F*\frac{1}{\pi z}\in \Sc'(\C)$, the convolution being
well-defined because $F$ has compact support. Since $\frac{1}{\pi z}$ is the fundamental solution of
the Cauchy-Riemann operator $\frac{\partial}{\partial\bar{z}}$, we have $\frac{\partial G}{\partial
\bar{z}}=F$ (cf., for example, \cite{Ho}, Theorem 1.2.2).  By the ellipticity of the Cauchy-Riemann
operator, $\singsupp G\subset\singsupp F\subset \supp F$, in particular, this means that $G$ is a
smooth function outside $\psupp F$, moreover,  $G$ is analytic outside $\psupp F$ (by $\singsupp F$
we denote the singular support of the  distribution $F$, see, e.g., \cite{Ho}, the largest open set
where the distribution coincides with a smooth function).
  Additionally,
$G(z)=\langle F,\frac1{\pi(z-w)}\rangle =\pi^{-1}\sum_{k=0}^\infty z^{-k-1} \langle F,w^k\rangle=0$
if $|z|>R$ and $R$ is sufficiently large. By analyticity this implies $G(z)=0$ for all $z$  outside
$\psupp F$.\end{proof}

\begin{proof}[Proof of Theorem \ref{3.Th.Alexroz}]

The  distribution  $F$, as any distribution with compact support, is of finite order,
therefore it belongs to some Sobolev space, $F\in H^s$ for certain $s\in \R^1$. If $s\ge0,$ $F$ is a
function and must be zero by Luecking's theorem. So, suppose that $s<0$.

Consider the first $\rb+1$ columns in the matrix $\Ac(F)$, i.e.
\begin{equation}\label{columns}
    a_{kl}=(\Tc_Fz^k,z^l)=\langle \s F,z^k\bar{z}^l\rangle, l=0,\dots \rb;\  k=0,\dots.
\end{equation}
Since the rank of the matrix $\Ac(F)$ is not greater than $\rb$, the columns are linearly dependent, in
other words, there exist coefficients $c_0,\dots,c_\rb$ such that $\sum_{l=0}^\rb a_{kl}c_l=0$ for any
$k\ge 0$. This relation can be written as
\begin{equation}\label{ColPolyn}
    \langle F,z^k h_1(\bar{z}) \rangle=\langle h_1(\bar{z})F, z^k\rangle=0, \ \ h_1(\bar{z}) = \sum_{k=0}^\rb c_l\bar{z}^l.
\end{equation}
Therefore, the distribution $h_1(\bar{z})F\in H^s$
satisfies the conditions of Lemma \ref{Lem.compsup}
and hence there exists a compactly supported
distribution $F^{(1)}$ such that $\frac{\partial
F^{(1)}}{\partial \bar{z}}=h_1F$. By the
ellipticity of the Cauchy-Riemann operator, the
distribution $F^{(1)}$ is less singular than $F$,
$F^{(1)}\in H^{s+1}$. At the same time,
\begin{gather}\label{newmatrix}\nonumber
    \langle F^{(1)}, z^k\bar{z}^l \rangle=
    (l+1)^{-1}\langle F^{(1)},\frac{\partial z^k\bar{z}^{l+1 }}{\partial \bar{z}}
    \rangle\\ =(l+1)^{-1}\langle h_1(\bar{z})F, z^k\bar{z}^{l}\rangle =
    (l+1)^{-1}\langle F,
    z^k\bar{z}^{l}h_1(\bar{z})\rangle,
\end{gather}
and therefore the rank of the matrix $\Ac({F^{(1)}})$ does not exceed the rank of the matrix $\Ac(F)$.

We repeat this procedure sufficiently many (say, $ n=[-s]+1$) times and arrive at the distribution
$F^{(n)}$ which is, in fact, a function in  $L_2$, for which the corresponding matrix $\Ac({F^{(n)}})$ has finite rank. By Luecking's
theorem, this may happen only if $F^{(n)}=0$.

Now we go back to the initial distribution $F$. Since,
by our construction, $\frac{\partial F^{(n)}}{\partial
\bar{z}}=h_n(\bar{z})F^{(n-1)}$, we have that
$h_n(\bar{z})F^{(n-1)}=0$ and therefore $\supp
F^{(n-1)}$ is a subset of the set of zeroes of the
polynomial $h_n(\bar{z})$. On the next step, since
$\frac{\partial F^{(n-1)}}{\partial
\bar{z}}=h_{n-1}(\bar{z})F^{(n-2)}$,  we obtain that
$\supp F^{(n-2)}$ lies in the union of  sets of zeroes
of polynomials $h_{n-1}(\bar{z})$ and $h_n(\bar{z})$.
After having gone all the way back to $F$, we obtain
that its support is a finite set of points lying in
the union of zero sets of polynomials $h_{j}$. A
distribution with such support must be a linear
combination of $\d$ - distributions in these points
and their derivatives, $F=\sum L_q\d(z-z_q)$, where
$L_q=L_q(D)$ are some differential operators. Finally,
to show that the number of points $z_q$ does not
exceed $\rb$, we construct for each of them the
interpolating polynomial $f_q(z)$ such that
$L_q(-D)|f_q|^2\ne0$ at the point $z_q$   while at the
points $z_{q'},\ q'\ne q$, the polynomial $f_q$ has
zero of sufficiently high order, higher than the order
of $L_{q'}$, so that $L_{q'}(f_q g)(z_{q'})=0$ for any
smooth function $g$. With such choice of polynomials,
the matrix with entries $\langle
F,f_q\overline{f_{q'}}\rangle $ is the diagonal matrix
with nonzero entries on the diagonal, and therefore
its size (that equals the number of the points $z_q$)
cannot be greater than the rank of the whole matrix
$\Ac(F)$, i.e., cannot be greater than $\rb$.
\end{proof}
\begin{remark}\label{3.Remark}The attempt to extend directly the original proof of Theorem \ref{3.Thm.Luecking} to the
distributional case would probably meet certain complications.  The following property is crucial in this proof: the algebra generated by polynomials of
the form $P_1(Z)\overline{Q_1(Z)}$ with symmetric $P_1,Q_1$ is dense (in the sense of the uniform convergence on compacts) in the
space of symmetric continuous functions. This latter property is proved above by a reduction to the Stone-Weierstrass theorem.

Now, if $F$ is a distribution that is not a measure, the analogy of reasoning in the proof
would require a similar density property, however not in the sense of the uniform convergence on
compacts, but in a stronger sense, the uniform convergence together with derivatives up to some fixed
order (depending on the order of the distribution $F$.)  The Stone-Weierstrass theorem seems  not to
help here since  it deals with  the uniform convergence only. Moreover, the required more general density
statement itself is \emph{wrong}, which follows   from the  construction below (see \cite{AlexRoz}).\end{remark}

\begin{example}\label{example} The algebra generated by the functions having  the form
$P_1(Z)\overline{Q_1(Z)}$, where $P_1,Q_1$ are symmetric polynomials of the variables
$Z=(z_0,\dots,z_N)$ is not dense in the sense  of the uniform  $C^l$-convergence  on compact sets  in
the space of $C^l$-differentiable symmetric functions, as long as $l\ge N(N-1)$.
To show this,  consider the differential operator $V(D)=\prod_{j<k}(D_j-D_k)$, $D_j=\frac{\partial}{\partial {z}}$. It is easy to check that
$V(D)H$ is symmetric for any antisymmetric function $H(Z)$ and $V(D)H$ is antisymmetric for any
symmetric function $H(Z)$. Further on, consider any function $H(Z)$ of  the form
$H(Z)=P_1(Z)\overline{Q_1(Z)}$ where $P_1(Z),Q_1(Z)$ are analytic polynomials. If at least one of
them is symmetric, we have
\begin{equation}\label{sympol}
    (V(D)V(\bar{D}))H(0)=0.
\end{equation}
In fact,$ V(D)V(\bar{D}) P_1(Z)\overline{Q_1(Z)} =[(V(D)P_1(Z)][\overline{V({D}){Q_1(Z)}}]$. In the
last expression, for  the symmetric polynomial $P_1$ , the corresponding polynomial $V(D)P_1(Z)$ is
antisymmetric, and therefore equals zero for $Z=0$. Now consider the symmetric function
$|V(Z)|^2=V(Z)\overline{V(Z)}$. We have
 $$V(D)V(\bar{D})V(Z)\overline{V(Z)} =[V(D)V(Z)][V(\bar{D})\overline{V(Z)}].$$
Now note that $V(Z)=\sum_\kappa C_{\kappa}\prod z_j^{\kappa_j}$ where the summing goes over
multi-indices $\kappa=(\kappa_1,\dots,\kappa_N)$, $|\kappa|=N$ and not all of real coefficients
$C_\kappa$ are zeros. Simultaneously, $V(D)=\sum_{\kappa}C_{\kappa}\prod D_j^{\kappa_j}$ with the
same coefficients. We recall now  that $\prod D_j^{\kappa_j}\prod z_j^{\kappa'_j}=0$ if $|\kappa|=|\kappa'|,$ $\kappa\ne
\kappa'$  and it equals $\kappa!$ if $\kappa = \kappa'$. Therefore,
$V(D)V(Z)=\sum_{\kappa}C_{\kappa}^2\kappa!$ is a positive constant.  In this way we have constructed
the differential operator $V(D)V(\bar{D})$  of order $N(N-1)$, satisfying \eqref{sympol} for any
function of the form $H(Z)=P_1(Z)\overline{Q_1(Z)}$ with symmetric $P_1,Q_1$, and  not vanishing on
some symmetric differentiable function $|V(Z)|^2$. Therefore the function $|V(Z)|^2$ cannot be
approximated by linear combinations of the functions $H(Z)=P_1(Z)\overline{Q_1(Z)}$ in the sense of
the uniform $C^{N(N-1)}$ convergence on compacts.\end{example}

D. Luecking's theorem was extended in a different direction by T.Le in \cite{Le}.
For the particular system of functions $f_k=z^k$ used above for the construction of the matrix $\Ac$, it turns out that its rank may even be infinite, but the assertion of the theorem still holds, as long as the range of the operator avoids sufficiently many analytical functions.

We will say that the set of indices $\Jc=\{n_j\}\subset\Z_+=\{n\in\Z,n\ge0\}$ is sparse if the series $\sum_{n\in \Jc} (n+1)^{-1}$ converges.

\begin{theorem}\label{3.LeToepl}(\cite{Le})Suppose that  $\m$  is a regular complex Borel measure and  that $\Jc=\{n_j\}\subset\Z_+$ is sparse, $\Jc'=\Z_+\setminus \Jc$. Consider the reduced matrix $\Ac^\Jc$ consisting of $a_{jk}: \ j, k\in\Jc' $. Suppose that the rank $\rb$ of $\Ac^\Jc$ is finite. Then the support of  $\m$ consists of no more than $\rb+1$ points. \end{theorem}

The original formulation of this theorem in \cite{Le} is given in  the terms of the Toeplitz operator itself. Denote by $\Ms,\Ns$ the space of polynomials spanned by monomials $z^j$ with $j$, respectively, in $\Jc,\Jc'$. In  Theorem \ref{3.LeToepl} it is supposed that the operator $\Tc_\m$, being restricted to $\overline{\Ns}$, has range in the linear span of $\overline{\Ms}$ and some finite-dimensional subspace in $\Ns$.

In the next section we will establish   a   result that generalizes Theorem \ref{3.LeToepl} in three directions: the multidimensional case will be considered, any distribution with compact support will replace the measure $\m$, and the condition of sparseness will be considerably relaxed.

\section{The multidimensional case}
In this Section we extend our main Theorem \ref{3.Th.Alexroz} to the case of Toeplitz operators in Bergman
spaces of analytical functions of several variables. For the case of a measure acting as weight,
there exist two ways  of proving this result, in \cite{BRChoe1} and \cite{RShir}, \cite{AlexRoz}. The first approach
generalizes the one used in \cite{Lue2} in proving Theorem \ref{3.Thm.Luecking}, the other one uses the induction on dimension. As it
follows from Remark \ref{3.Remark}, for the case of distribution the approach of \cite{BRChoe1}
is likely to meet some complications. We present here the proof given in \cite{AlexRoz}, with some  modifications.
\begin{theorem}\label{4.Thm.Dim d} Let $F$ be a distribution in $\Ec'(\C^d)$. Consider the matrix
\begin{equation}\label{4.MatrtixDimd}
    \Ac(F)=(a_{\a\b})_{\a,\b\in \Z_+^d}; \ a_{\a\b}=\langle F, \zb^\a\bar{\zb}^\b\rangle, \ \zb=(z_1,\dots,z_d)\in \C^d.
\end{equation}
Suppose that the matrix $\Ac(F)$ has finite rank $\rb$. Then $\card\supp F\le \rb$ and $F=\sum L_q
\d(\zb-\zb_q)$, where $L_q$ are differential operators and $\zb_q$, $1\le q\le \rb$, are some points in $\C^d$.
\end{theorem}
We notice first, following \cite{RShir}, that if    the  function   $g$ is   analytical and bounded
in some polydisk neighborhood of $\supp F$  and
$F_{g}$ is the distribution $|g|^2F$ then
$\rank\Ac({F_g}) \le \rank \Ac(F)$.
To  show    this,   we denote by $M_g$ the bounded operator acting in $\Bc^2$ by multiplication by $g$.
The adjoint operator $M_g^*$ is, of course, bounded as well. Consider the quadratic form
of the operator $\Tc_{F_g}$: for $u\in \Bc^2$,
\begin{gather*}
    (\Tc_{F_g}u,u)=\langle F,gu \bar g\bar u\rangle=(\Tc_F(gu),gu)=(\Tc_F M_g u, M_g u)=(M_g^*\Tc_F M_g u,u).
\end{gather*}
So we see that the operator $u\mapsto \Tc_{F_g}u$ coincides with $M_g^*\Tc_F M_g$. The multiplication
by bounded operators does not increase the rank of an operator, and  the property follows.

 Thus, it is sufficient to prove the statement that is, actually, only formally
 weaker  than Theorem \ref{4.Thm.Dim d}.
 \begin{theorem}\label{4.TheorDimDmore} Suppose that for any function $g(\zb)$, analytic and bounded in a polydisk
 neighborhood of the
 support of the distribution $F$,  the conditions of Theorem  \ref{4.Thm.Dim d} are
 fulfilled with the distribution $F$ replaced  by $|g(\zb)|^2F\equiv F_g$. Then $\card\supp F\le \rb$ and
  $F=\sum_{1\le q\le \rb} L_q
\d(\zb-\zb_q)$, where $L_q$ are differential operators.\end{theorem}

\begin{proof}
 We use the induction on $d$.
 For $d=1$ the statement of Theorem \ref{4.TheorDimDmore} coincides with the one of
Theorem \ref{3.Th.Alexroz} that was proved in Sect. \ref{SectionLuecking}. We suppose that we have established our
statement in the    complex dimension $d-1$ and consider the $d$-dimensional case. We denote the variables as
$\zb=(z_1,\zb'), \ \zb'\in\C^{d-1}$.

 For a fixed function $g(\zb)$ we denote by $G(g)=\pi_*F_g$ the distribution in $\Ec'(\C^{d-1})$ induced from $F_g$
  by the projection $\pi:\zb\mapsto
 \zb'$: for $u\in C^\infty(\C^{d-1})$,
 \begin{equation}\label{4.proj}
    \langle G(g), u\rangle = \langle F_g, 1_{\C^1}\otimes u\rangle.
 \end{equation}
Although the function $g$ is defined only in a
polydisk, the distribution \eqref{4.proj} is well
defined since this polydisk contains $\supp F$.

Consider the submatrix $\Ac'({F_g})$ in the matrix $\Ac({F_g})$ consisting only of those
$a_{\a\b}=\langle |g|^2 F, \zb^\a\overline{\zb}^\b\rangle$  for which $\a_1=\b_1=0$. It follows from
\eqref{4.proj}, that the matrix $\Ac'({F_g})$ coincides with the matrix $\Ac({G(g)})$ constructed for the
distribution $G(g)$ in dimension $d-1$. Thus, the matrix $\Ac(G(g))$, being a submatrix of a finite
rank matrix, has a finite rank itself, moreover, $\rank\Ac({G(g)})\le \rb$. By the inductive assumption,
this implies that the distribution $G(g)$ has  finite support consisting of $\rb(g)\le \rb$ points
$\z_1(g),\dots,\z_{\rb(g)}$;  $\z_q(g)\in \C^{d-1}$ (the notation reflects the fact that both the points
and their quantity may depend on the function $g$). Among all functions $g$, we can find the one,
$g=g_0$, for which $\rb(g)$ attains its maximal value $\rb_0\le \rb$. Without losing in generality, we can
assume that $g_0=1$.

Fix an $\e>0$, sufficiently small, so that
$2\e$-neighborhoods of $\z_q(1)$ are disjoint, and
consider the functions $\varphi_q(\zb')\in
C^{\infty}(\C^{d-1})$, $q=1,\dots, \rb_0$ such that $\supp
\varphi_q $ lies in the $\e$-neighborhood of the point
$\z_q(1)$ and $\varphi_q(\zb')=1$ in the
$\frac\e2$-neighborhood of $\z_q(1)$.  We fix an
analytic function $g(\zb)$ and consider for any $q$ the
distribution $\F_q(t,g)\in \Ec'(\C^d),$
$\F_q(t,g)=|1+tg(\zb)|^2
\varphi_q(\zb')F=\varphi_q(\zb')F_{1+tg}$. For $t=0$,
$\F_q(t,g)=\varphi_q(\zb')F$,  the point $\z_q(1)$
belongs to the support of $\pi_* \F_q(0,g)$, and
therefore for some function $u\in C^\infty(\C^{d-1})$,
$\langle\pi_* \F_q(0,g),u\rangle\ne0$. By continuity,
for $|t|$ small enough, we still have $\langle\pi_*
\F_q(t,g),u\rangle\ne0$, which means that the
$\e$-neighborhood of the point $\z_q(1)$ contains at
least one point in the support of the distribution
$G(1+tg)$.  Altogether, we have not less than $\rb_0$
points of the support of $G(1+tg)$ in the union of
$\e$-neighborhoods of the points $\z_j(1)$. However,
recall, the support of $G(1+tg)$ can never contain
more than $\rb_0$ points, so we deduce that for $t$
small enough, there are no points of the support of
$G(1+tg)$ outside the $\e$-neighborhoods of the points
$\z_q(1)$,  for $|t|$ small enough (depending on
$g$.) Thus the support of the distribution $G(1+tg)$ is contained in the    $\e$-neighborhood   of  the  set of points $\z_j(1)$ for any $g$.

 Now we introduce a function $\p\in C^\infty(\C^{d-1})$ that  equals $1$ outside
$2\e$-neighborhoods of the points $\z_q(1)$ and vanishes in $\e$-neighborhoods of these points. By the above resoning, the distribution $\p G(1+tg)$ equals zero for any $g$, for $t$ small enough. In
particular, applying this distribution to the function $u=1$, we obtain
\begin{equation}\label{4.Ftg}
\langle\p G(1+tg),1\rangle=\langle\p F,|1+tg|^2\rangle=\langle \p F, 1+2t\re g+ t^2|g|^2\rangle=0.
\end{equation}
By the arbitrariness of $t$ in a small interval, \eqref{4.Ftg} implies that $\langle \p F,
|g|^2\rangle=0$ for any $g$. By standard polarization, this implies that for any  functions $g_1,g_2$
analytical in a polydisk neighborhood of $\supp F$.
    \begin{equation}\label{4.Fgg}
\langle\p F, g_1\overline{g_2}
    \rangle=0.
    \end{equation}

    Any  polynomial $p(\zb,\bar{\zb})$ can be represented as a linear combination of functions
    of the form $g_1\overline{g_2}$, so, \eqref{4.Fgg} gives
    \begin{equation}\label{4.FP}
\langle\p F, p(\zb,\bar{\zb})
    \rangle=0.
    \end{equation}
Now  we take any function $f\in C^\infty(\C^d)$ supported in the neighborhood $U$ of $\supp F$ such
that $f=0$  on the support of $\p$. We can approximate $f$
 by polynomials of the form $p(\zb,\bar{\zb})$ uniformly on $\overline{U}$ in the sense of $C^l$, where
 $l$ is the order of the distribution $F$. Passing to the limit in \eqref{4.FP}, we obtain $\langle\p F,
 f \rangle=\langle F,
 f \rangle=0.$

 The latter relation shows that $\supp F\subset \bigcup_q\{\zb: |\zb'-\z_q(1)|<2\e\}$. Since $\e>0$ is arbitrary,
 this  implies that  $\supp F$
lies in the union of affine subspaces $\zb'=\z_j$, $j=1,\dots,\rb_0$ of complex dimension $1$.

 Now we
repeat the same reasoning having chosen instead of $\zb=(z_1,\zb')$ another decomposition of the complex
variable $\zb$: $\zb=(\zb'',z_d)$. We obtain that for some points $\x_k\in \C^{d-1}$, no more than $\rb$ of
them, the support of $F$ lies in the union of subspaces $\zb''=\x_k$. Taken together, this means that,
actually, $\supp F$ lies in the intersection of these two systems of subspaces, which consists of no
more than $\rb^2$ points $\zb_s$. The number of points is finally reduced to $\rb_0\le \rb$ in the same way
as in Theorem\ref{3.Th.Alexroz}, by choosing a special system of interpolation functions.
\end{proof}

The theorem just proved can be extended to the case of a    sparse range, following the pattern of Theorem \ref{3.LeToepl}.

\begin{definition}\label{4.Definsparse}Let $\Jc\subset\Z_+^d$ be a set of multiindices and $\g\in\Z_+^d$ be a fixed multiindex. We say that the set $\Jc$ is $N$- sparse in the   direction $\g$ if for any $\a\in\Z_+^d$,
\begin{equation}\label{4.SparseDef}
    \limsup_{n\to\infty} n^{-1}\#\{(\a+([0,n]\g))\cap \Jc\}<N^{-1}.
\end{equation}
\end{definition}
In other words, the fact that the set $\Jc$ is sparse in direction $\g$ means that along any half-line starting at some point of $\Z_+^d$ and going in the direction $\g$, the density of the points of $\Jc$ on this line is less than $N^{-1}$.

For a multiindex $\g=(k_1,\dots,k_d)$ we denote by $\nb(\g)$ the set of indices $j$ such that $k_j=0$. We introduce the set $\Jc'=\Z_+^d\setminus \Jc$ and the reduced matrix $\Ac^\Jc(F)$ consisting of $a_{\a,\b}: \ \a, \b\in\Jc' .$

\begin{theorem}\label{4.LeMulti}Suppose that for a distribution $F\in\Ec'(\C^d)$  the reduced matrix $\Ac^\Jc(F)$ has finite rank $\rb$, and for some $\e>0$, the set $\Jc$ is $2\rb+2+\e$-sparse in some directions $\g_l$, $l=1,\dots,\lb$ such that
\begin{equation}\label{4.direct}
    \cup \nb(\g_l)=\{1,\dots, d\}.
\end{equation}
Then the distribution $F$ has finite support consisting of no more than $\rb+1$ points. In the particular case when $F$ is a function, the condition \eqref{4.direct} can be dropped and if $\Jc$ is $2\rb+2+\e$-sparse just in one, arbitrary, direction then $F=0$.
\end{theorem}
\begin{remark}For the case of measures, a version of Theorem \ref{4.LeMulti} with a more restrictive  notion of sparseness,  has been proved in \cite{Le1}.
\end{remark}
\begin{remark}The condition \eqref{4.direct} is sharp in the sense that if it is violated, the support of $F$ can be infinite.
 Consider the set $\Jc$ consisting of multiindices having zero in the first position.
 This set $\Jc$ is $N$-sparse for any $N$ in any direction
 $\g_l$, with a non-zero in the first position. For such directions $\g_l$, we have $1\not\in{\cup \nb(\g_l)}$.
  Let the distribution $F\in\Ec'(\C^d)$ be a measure supported in the subspace $\{z_1=0\}$.
  Then for any $\a,\b\not\in\Jc$, the functions $\zb^\a,\overline{\zb^\b}$ vanish on $\{z_1=0\}$ and thus $\langle F, \zb^\a \overline{\zb^\b}\rangle=0,$ so the statement on the finiteness of support becomes wrong. Of course, due to the second part of Theorem \ref{4.LeMulti}, such examples are impossible in the case when the distribution is, in fact, a function.
\end{remark}
\begin{proof}The proof follows the structure of the proof of Theorem \ref{3.LeToepl} in \cite{Le}, with two main ingredients replaced by their multi-dimensional analogies and with  the extension to distributions and a    more general notion of sparse sets.

First, similar to Lemma \ref{3.LueLem1}, the following two properties are equivalent:
\begin{enumerate}\item
 the matrix $\Ac(F)$ has finite rank not greater than $\rb$;
 \item for any collections of $2N=2\rb+2$ multiindices $\a_0,\dots,\a_\rb$, $\b_0,\dots,\b_\rb$
     \begin{equation}\label{4.VanishingMod}
     \textstyle
     \f^{\otimes N}\blp \prod_{i\in(0,\rb)} \zb_i^{\a_i} \det \overline{\zb_i}^{\b_l} \brp = 0.
     \end{equation}
\end{enumerate}
The proof of this fact is quite similar to the proof of  Lemma \ref{3.LueLem1}; one can also see details in \cite{BRChoe1}, P. 215.

Now  we fix the multiindices $\a_0,\dots,\a_\rb$, $\b_0,\dots,\b_\rb$ and, supposing that the set $\Jc$ is $2\rb+2+\e$ - sparse in the  direction $\g$, consider the set of multiindices
\begin{equation}\label{4:Zs}
    \Zs=\Z_+^d\setminus\blp (\cup_{j=0}^\rb(\Jc-\a_j))\bigcup(\cup_{j=0}^\rb(\Jc-\b_j))\brp.
\end{equation}
The set $\Z_+^d\setminus\Zs$ thus consists of $2N$ shifts of the set $\Jc$,  therefore
\begin{equation*}
    \liminf_{n\to\infty}\left\{n^{-1}\#\{\Zs\cap\g[0,n]\}\right\}>\e>0.
\end{equation*}
In other words, the set of integers $n$ such that $\a_j+n\gamma, \b_j+n\gamma\not\in\Jc$ for \emph{all} $j$ has positive density in $\Z_+$. In particular,  this    means   that
\begin{equation}\label{4:Zs1}
    \sum_{n:\g n \in \Zs}(n+1)^{-1}=\infty.
\end{equation}

Now we consider the function of the complex variable $w$:
\begin{gather*}%\label{4:LeFunction}
    \F(w)=\langle F^{\otimes N},\prod \zb_j^{\a_j+\g w}\overline{\det\blp \zb_j^{\b_k+\g w} \brp}\rangle\\=\langle F^{\otimes N},\prod \zb_j^{\a_j}\overline{\det\blp \zb_j^{\b_k} \brp}|\prod \zb_j^\g|^{2Nw}\rangle.
\end{gather*}
The function onto which the distribution $F^{\otimes N}$ acts, is not smooth for non-integer $w$, but we will take care of this in the following way. The distribution $F$, having compact support, must have finite order, $\k\ge0$. Thus $F$ can be extended as a functional on $\k$ times differentiable functions. The function $|\prod \zb_j^\g|^{2Nw}$ belongs to $C^{\k}$ for $2N\re w\ge \k$, so,  in the half-plane $\K=\{\re w>\frac{\k}{2N}\}$ the function $\F(w)$ is well defined, analytical, and it is continuous in $\overline{\K}$. Therefore, if the support of $F$ lies in the ball $|\zb|<R$, the function $\Psi(w)=R^{-Nw}\F(w)$ is a bounded analytical function in $\K$.
Since, by our construction, all $\a_j+\g n$, $\b_n+\g n$ belong to $\Z_+\setminus\Jc$ for all $n\in \Zs$, we have
 \begin{equation}\label{4.LeProof}\F(n)=\Psi(n)=0,\  n\in \Zs.\end{equation}
 Now let $H(\z)=\Psi\blp\frac{1+(\z+\k)}{1-(\z+\k)}\brp$.
  Then $H$  is a bounded analytical function on the unit disk. For any $n\in \Zs$, the equation \eqref{4.LeProof} implies that $H(\frac{n-1-\k}{n+1+\k})=0$. Now
  \begin{equation*}
    \sum_{n\in\Zs}\blp 1-\frac{n-1-\k}{n+1+\k}\brp=\sum_{n\in\Zs}\frac{2\k+2}{n+1+\k}=\infty
  \end{equation*}
  by \eqref{4:Zs1}.
The Corollary to Theorem 15.23 in \cite{Rudin} shows that in this case $H$ should be indentically zero on the unit disk,
and therefore $\F(w)=0, \re w>k$. By continuity, $\F(\k)=0$, and this means, in particular that
\begin{equation*}%\label{3.Le9}
    (|\zb^\g|^{2\k}F)^{\otimes N}\blp\prod_{j=0}^\rb \zb_j^{\a_j}\overline{\det(\zb_j)^{\b_k}} \brp=0.
\end{equation*}
In the reasoning above, the multiindices $\a_j,\b_j$ are arbitrary, this means that \eqref{4.VanishingMod} holds for the distribution $|\zb^\g|^{2\k}F$, and therefore,  by  Theorem \ref{4.Thm.Dim d},   this distribution must have a  support consisting of a finite number of points. Therefore the support of $F$   itself   is contained in the union of the above points and the subset where $\zb^\g=0$. The latter subset is the union of the subspaces $z_m=0$ for those $m$ that do not belong to $\nb(\g)$. In particular, if $F$ is a function, this implies that $F=0$. In the general case, we repeat the reasoning in the proof for any direction $\g_l$. Since, by the conditions of the theorem,  $\cup\nb(\g_l)=\{1,\dots,d\}$, the intersection of zero sets of $\zb^{\g_l}$ consists only of the point $0$, and this proves our statement.
\end{proof}
\section{Applications}\label{Applic}
In this section we give some applications of the results on the finite rank Toeplitz operators in analytical Bergman spaces.

\subsection{Approximation}\label{Appr}

For a subset $Q\subset C(\O)$, we denote by $\Zc(Q)$ the  set of common zeros of functions in $Q$. Conversely,
for a subset $E$ of $\O$, we denote by $J(E)$ the ideal in $C(\O) $ consisting of all functions
 vanishing on $E$.
Given a subspace $W$ in the Bergman space $\Bc(\O)$ of analytical functions, we denote by $\widehat{W}$ the closure in $C(\O)$ of the span of functions of the form $h(z)=f\bar{g}, f\in\Bc(\O), g\in W$ in the topology of uniform convergence on compacts in $\O$.
\begin{theorem}\label{5.ChoeAppl}(\cite{BRChoe1}). Let $W$ be a subspace in $\Bc(\O)$ with finite codimension. Then $\Zc(W)$ is a finite set and $\widehat{W}=J(\Zc(W))$. In particular, if $\Zc(W)=\varnothing$ then  $\widehat{W}=C(\O)$.\end{theorem}
\begin{proof} Endowed with the topology of uniform convergence on compact
sets, the space $C(\O)$ is locally convex and its continuous linear functionals are
identified with complex Borel measures supported on compact sets in $\O$. Let $Y$ be the space of measures orthogonal to $\widehat{W}$.
 If $Y\ne\varnothing$ there should exist
 a complex Borel measure $\m\ne0$ supported on a compact set in $\O$ such that
\begin{equation}\label{5.Choe.1}
    0=\int_{\O}f\bar{g}d\m=\int_\O(\Tc_\m f)\bar{g}d\l
\end{equation}
for all $f\in\Bc, g\in W$. This shows that $T_\m \Bc$
is contained in $W^\bot$, which is finite dimensional. By Theorem \ref{3.Th.Alexroz}, $\m$ must be supported on some finite set in $\O$, say $E_\m$. It follows that $\Tc_\m \Bc$ is spanned by finitely many kernel functions $P(\cdot,w)$, $w\in E_\m$. By \eqref{5.Choe.1}, we have that these functions $P(\cdot,w)$ lie in $W^\bot$. Since these functions are also linearly independent, the union $E$ of the sets $E_\m, \ \m\in Y$, must be finite. By the reproducing property, $E\subset \Zc(W)$. Moreover, we have $E=\Zc(W)$, because point masses at the points of $\Zc(W)$ belong to $Y$.
Now, since each $\m\in Y$ is supported in $E$, the ideal $J(E)=J(\Zc(W))$ is annihilated by all $\m\in Y$. Thus $J(\Zc(W))\subset \widehat{W}$, and the converse inclusion is obvious.  The case $Y=\varnothing$ is easily treated by the fact
that $Y=\varnothing$ if and only if $\Zc(W)=\varnothing$, which one may see from the proof above.\end{proof}

Theorem \ref{5.ChoeAppl} can be understood as saying that the linear combinations of the functions of the form $f\bar{g}$, $f\in W$, $g\in \Bc$, can approximate any continuous function uniformly on any compact not containing points from some finite set. By means of the more general Theorem \ref{4.LeMulti}, we can extend this approximation result in several directions.

\begin{theorem}\label{5:Appr} Let $\O\subset \C^d$ let  $\Jc\subset \Z_+^d$ be some set of multi-indices, $\Jc'=\Z_+^d\setminus \Jc$, satisfying the conditions of Theorem \ref{4.LeMulti} with $\rb<2N+1$, for some $N$. Denote by $\Pc=\Pc(\Jc)$ the space of polynomials of the form $\pb(z)=\sum c_\a \zb^\a, \a\in \Jc'$,  Let $U,V$ be linear subspaces in $ \Pc(\Jc)$ with codimension not greater than $N$. Then there are no more than $2N+1$ points  $\wb_\varkappa\in \O$ such that  for any $n$, the space $\Rc$ of linear combinations of functions of the form $\pb(\zb)\overline{\qb(\zb)}, \ \pb(\zb)\in U,  \qb(\zb)\in V$, is dense in $C^n(\O)$ in the sense of uniform convergence  of all derivatives of order not higher than $n$, on any compact $K\subset\O$ not containing the points $\wb_\varkappa$.\end{theorem}

Compared with Theorem \ref{5.ChoeAppl}, this theorem takes care of  a more  strong  type of convergence, while the approximating set is considerably smaller.

\begin{proof}Suppose that on some compact $K$ the functions $\pb(\zb)\overline{\qb(\zb)}$ are not dense in the sense of uniform convergence on $K$ with derivatives of order up to $n$. This means that there exists a distribution $F$ with support in $K$ such that $\langle F,\pb(\zb)\overline{\qb(\zb)}\rangle=0$ for all $ \pb(\zb)\in U,  \qb(\zb)\in V$.

 Since $V$ has finite codimension in $\Pc(\Jc)$, there exist no more than $N$ polynomials $\f_0,\dots, \f_{N_0}$, $N_0< N$, so that $\Pc(\Jc)=V+\Span(\f_0,\dots,\f_{N_0})$.  Similarly, there exist no more than $N$ polynomials $\p_0,\dots,\p_{N_1}$, $N_1< N$, so that $\P(\Jc)=U+\Span(\p_0,\dots,\p_{N_1})$.

 We choose some basis $\pb_i(z)$ in $U$ and some basis $\qb_j(z)$ in $V$. Consider the infinite matrix $\Cc_0$
  consisting of elements $b_{ij}=\langle F,\pb_i(z)\overline{\qb_j(z)}\rangle$, which are, of course, all zeros. Now, we append the matrix $\Cc_0$ by $N_0$ columns $b_{i,-s}=\langle F,\pb_i(z)\overline{\f_s(z)}\rangle$, $s=0,\dots,N_0$ and then by $N_1$ horizontal rows $b_{-t,j}=\langle F,\p_t(z)\overline{\qb_j(z)}\rangle$, $b_{-t,-s}=\langle F,\p_t(z)\overline{\f_s(z)}\rangle$, $t=0,\dots,N_1$, thus obtaining the matrix ${\Cc}$. Each of these two operations increases the rank of the matrix no more than by $N$, so $\rank({\Cc})\le 2N.$ Now, since any monomial $\zb^k$, $k\in\Jc'$ is a finite linear combination of polynomials $\pb_i, \p_t$ and any monomial $\zb^l,$ $l\in \Jc'$ is a finite linear combination of polynomials $\qb, \f_s$, the matrix $\Ac^{\Jc}(F)$ consisting of $\ab_{k,l}=\langle F, \zb^k\overline{\zb^l}\rangle$, $k,l\in \Jc',$ has rank not greater than $\rank({\Cc})$, i.e., $\rank(\Ac^{\Jc}(F))\le 2N.$

 Now  Theorem \ref{4.TheorDimDmore} implies that the distribution $F$ has support consisting of no more than $2N+1$ points; we denote this set $\Zc(F)$. For some other distribution $G$, also vanishing on all functions of the form $\pb(\zb)\overline{\qb(\zb)}, \ \pb(\zb)\in U,  \qb(\zb)\in V $, the support $\Zc(G)$, by the same reasoning also consists of no more than $2N+1$ points. By considering a linear combination of $F$ and $G$, we see that still $\#\{\Zc(F)\cup\Zc(G)\}\le 2N+1$. So, the support of any distributions vanishing on $\pb(\zb)\overline{\qb(\zb)}, \ \pb(\zb)\in U,  \qb(\zb)\in V ,$ is a part of some set $E\subset K$, that has no more than $2N+1$ points. Therefore,  any function in $C^{n}$ can be approximated by the functions of the form $\pb(\zb)\overline{\qb(\zb)}, \ \pb(\zb)\in U,  \qb(\zb)\in V ,$ in $C^{n}$ uniformly on any compact $K'\subset K$, not containing these points. Finally, the compact $K$  in our construction can be chosen arbitrarily large, while the set $E$  can never have more than $2N+1$ points and thus can be taken independently of $K$.
\end{proof}

We give an example of the application of Theorem \ref{5:Appr}. For the sake of simplicity, we take $\Jc=\varnothing$. For multiindices $\a,k\in \Z_+^d$ we write $\a\prec k$ if each component of $\a$ is not greater than the corresponding component of $k$, while at least one component is strictly less.
  \begin{example}
 Suppose that for any multi-index   $k\in\Z_+^d$, $|k|>m$, two polynomials $\pb_k(\zb)$ and $\qb_k(\zb) , \zb\in \C^d$ are  given, of the form $\pb_k(\zb)=\zb^k+\sum_{\a\prec k}c_{\a,k}\zb^\a$, resp.,  $\qb_k(\zb)=\zb^k+\sum_{\b\prec k}p_{\b,k}\zb^\b$. These sets of polynomials have codimension not greater than $N=\binom{d+m}{d}$ in the space of all polynomials. Thus, Theorem \ref{5:Appr} guarantees that there are no more than $2N+2$ points $\wb_\varkappa$ such that any $C^n-$ function can be approximated by linear combinations of $\pb_k(\zb)\overline{\qb_l(\zb)}$ on any compact not containing these points.
\end{example}

\subsection{Products of Toeplitz operators}
It has  been known since long ago that the product of two Toeplitz operators in the Hardy space on a circle can be zero only in the case one of them is zero, see \cite{BrowHal}. This result has been gradually extended to  an arbitrary finite product  of operators on a circle (see \cite{AlemVuk}) and to  the multi-dimensional case, i.e. operators in  the Hardy space on the torus, where the product of up to six Toeplitz operators is taken care of, see \cite{Ding}.

Much less understandable is the situation with  Toeplitz operators in the Bergman space; even for the case of a disk it is still not    known if it is true that
for $f,g\in L_\infty$, the relation $\Tc_f\Tc_g=0$, or, more generally  $\rank(\Tc_f\Tc_g)<\infty$,  implies vanishing of $g$ or $f$. Affirmative answers to this problem, as well as to its multidimensional versions,
have been obtained  only in rather special cases, say, under the assumption that the functions $f,g$ are harmonic or $\mb-$ harmonic (see \cite{Guo}, \cite{BRChoe2} where extensive references can also be found.)

We present here some very recent results on the finite rank product problem, essentially obtained by T. Le \cite{Le}.

\begin{theorem}\label{Le.Product.2functions}
Let $D$ be the unit disk in $\C^1$. Suppose that the function $f(z)\in L_2(D)$, $|z|\le 1$ has in the expansion in polar coordinates the form
\begin{equation}\label{5.LePr2.1}
    f(re^{i\theta})=\sum_{n=-\infty}^{M}f_n(r)e^{in\theta},
\end{equation}
and  $\hat{f}_M({l})=\int_0^1 f_M(r)r^ldr\ne 0$ for all  $l$ large enough, $l>l_0$. If for some distribution $G\in \Ec'(D)$, the product $\Tc_G\Tc_f$ has finite rank, the distribution $G$ must have finite support. In particular, if $G$ is a function then $G=0.$
\end{theorem}
\begin{proof} The proof follows mostly the one in \cite{Le}, with modifications allowed by more advanced finite rank theorems. First, recall that the Bergman space $\Bc^2$ on the disk has a natural orthonormal  basis
 \begin{equation}\label{5:base}\eb_s(z)=\sqrt{s+1}z^s, \ s=0,1,\dots.\end{equation}
  The matrix representation of the operator $\Tc_f$ in this basis has the form
\begin{equation*}
    (\Tc_f\eb_k,\eb_l)=C_{k,l}\int\limits_0^1f_{l-k}(r)r^{k+l+1}, \ k,l\ge0,
\end{equation*}
 $C_{kl}=2\sqrt{(k+1)(l+1)}.$ By our  assumption about $f$, we have  $(\Tc_f\eb_k,\eb_l )=0$ whenever $l-k<M$. Thus for $k\in \Z_+$, we can write
\begin{gather*}
    \Tc_f\eb_k=\sum_{l=0}^{k+M}(\Tc_f\eb_k,\eb_l)\eb_l=
    C_{k,k+M}\hat{f}_M(2k+M+1)\eb_{k+M}+\\
    \sum_{l=0}^{k+M+1}C_{k,l}\hat{f}_{l-k}(k+l+1)\eb_l.
\end{gather*}
This shows that when $k+M\ge 1$ and $2k+M+1>l_0$,
 the function $\eb_{k+M}$ can be expressed as a linear combination of $\Tc_f\eb_k$ and $\eb_l, l<M+k$.
 Now suppose that $\Tc_G\Tc_f$ has finite rank $\rb$ and let $\f_1,\dots,\f_\rb$ be some basis in the range of $\Tc_G\Tc_f$. Then for any nonnegative integer $k$ such that $k+M\ge1$ and $2k+M+1>l_0$,  the function $\Tc_G \eb_{k+M}$ is a linear combination of $\f_1,\dots,\f_\rb$ and $\Tc_G\eb_l,$ $l\le k+M$. We substitute consecutively this expression for $\Tc_G\eb_{k+M}$ into the similar expression for $\Tc_G\eb_{k'+M}$, for $k'>k$. Thus all functions $\Tc_G \eb_{k'+M}$, $k'+M>1$, $2k'+M+1>l_0$,  will be expressed as  linear combinations of functions $\Tc_G\eb_{k+M}$ with $2k+M+1\le l_0$ and the finite set of functions $\f_1,\dots,\f_\rb$. This means that the matrix with entries $(\Tc_G \eb_{k},\eb_{l})$ has finite rank.
 Now we can apply Theorem \ref{4.LeMulti} that grants the required properties for $G$.
  \end{proof}

  Since for any monomial $z^k\bar{z}^l$, $\widehat{r^{k+l}}(m)$ is never zero, the conditions of Theorem \ref{Le.Product.2functions} are fulfilled for any $f$ having the form $f(z)=p(z,\bar{z})+\overline{h(z)}$ where $p$ is a nonzero polynomial of $z,\bar{z}$ and $h$ is a bounded analytical function.

  Another type of results on finite rank products of Toeplitz operators in the analytical Bergman space in the unit disk or polydisk, established in \cite{Le}, \cite{Le1}, covers the case when  all Toeplitz weights, except one, are functions of a special form. We present here the formulation of the general theorem proved in \cite{Le1}, generalized to cover the case of of distributional weights.

  \begin{theorem}\label{5:ProdRadial}Let $f_1,\dots,f_{m_1+m_2}$ be bounded functions in the polydisk $D^d\subset\C^d$ such that each of them is radial, $f_j(z_1,\dots,z_d)=f_j(|z_1|,\dots,|z_d|)$ and none is identically zero. For a collection of multiindices $\a_j,\b_j\in \Z_+^d, \ j=1,\dots, m_1+m_2$ we set $g_j(z)=f_j(z)\zb^{\a_j}\bar{\zb}^{\b_j}$.
  Suppose that $F$ is a distribution with compact support in $D^d$ and the  operator
  \begin{equation}\label{5:LeProduct}
    A=\Tc_{g_1}\dots \Tc_{g_{m_1}}\Tc_F \Tc_{g_{m_1+1}}\dots \Tc_{g_{m_1+m_2}}
  \end{equation}
  has finite rank. Then $F$ has finite support. In particular, if $F$ is a function, $F$ is zero.
  \end{theorem}
  The proof is based upon the consideration of the kernel of the product of the operators $\Sc_1=\Tc_{g_1}\dots \Tc_{g_{m_1}}$ and the range of $\Sc_2 =\Tc_{g_{m_1+1}}\dots \Tc_{g_{m_1+m_2}}$. The action of these operators is explicitly described in the natural basis in the Bergman space (and it is here the geometry of the polydisk is crucial.) The set of multiindices numbering the basis functions in the kernel of $\Sc_1$ and in the cokernel of $\Sc_2$ turns out to be sparse. Therefore, the finiteness of the rank of $\Sc_1 \Tc_F \Sc_2$ leads to the finiteness of the rank of the properly restricted operator $\Tc_F$. The reasoning concludes by the application of Theorem \ref{4.LeMulti}.

  To demonstrate the idea, not going into complicated details, we present the proof, borrowed from \cite{Le}, \cite{Le1}, for the most simple case, when $d=1$, $m_1=1$, $m_2=0$,
   so the operator $A$ in \eqref{5:LeProduct} has the form $A= T_gT_F $
\begin{proof}As in the proof of Theorem \ref{Le.Product.2functions}, we consider the standard orthogonal  basis $
\eb_s$ in the Bergman space, given by \eqref{5:base}. In this base, the action of the operator $\Tc_g$ for $g(z)=f(|z|)z^{\a}\bar{z}^\b$ is easily calculated,
\begin{equation}\label{5:ActionTg}
    \Tc_g \eb_s=\left\{ \begin{array}{c}
              0,\ s<\a-\b \\
              C\hat{f}(2s+2\a+1)\eb_{s+\a-\b},\ s\ge \a-\b
            \end{array}\right. ,
\end{equation}
with some positive constants $C$, depending on all indices and exponents in \eqref{5:ActionTg}.

Denote $\Jc=\{s:s<\a+\b\}\bigcup\{s :\hat{f}(2s+2\a+1)=0\}.$
Since the function $f$ is nonzero,  the set $\Jc$ is sparse by M\"{u}ntz-S\'{a}szs theorem. For $s\not\in \Jc$, we see from  \eqref{5:ActionTg} that $\Tc_g\eb_s\ne 0$ and $\eb_{s+\a-\b}$ is a multiple of $\Tc_g\eb_s$. Suppose that $\ff\in \Bc$ is a function such that $\Tc_g\ff=0.$
Then
\begin{equation*}
    0=\Tc_g\ff=\Tc_g\left(\sum_s(\ff,\eb_s)\eb_s\right)=\sum_s(\ff,\eb_s)\Tc_g\eb_s.
\end{equation*}
By \eqref{5:ActionTg}, this implies that $(\ff,\eb_s)=0$ for all $s\not\in \Jc$. Therefore, $\Ker \Tc_g$ is contained in the closed  span  of $\{\eb_s,\ s\in\Jc\}$. It follows that the rank of the matrix $(\Tc_F \eb_s,\eb_t), \ s,t\not\in \Jc$ is not greater than the rank of $\Tc_g\Tc_F$, thus it is finite. Finally, Theorem \ref{4.LeMulti} applies.
\end{proof}
\subsection{Sums    of  products of Toeplitz operators}
 Another interesting problem in the theory of Bergman spaces consists in determining the condition for some algebraic expression involving  Toeplitz operators to be a Toeplitz operator again. The results existing by now concern only Toeplitz operators with weights of some special form.

In \cite{BRChoe3} this problem has been considered in the following setting. Suppose  that $u_j,v_j,\ j=1,\dots,n,$ and $w$ are $\mb$--harmonic functions in the polydisk  $D^\mb\subset{\C^\mb}$.
\begin{theorem}(\cite{BRChoe3}) The necessary and sufficient condition for the  the operator $S=\Tc_w+\sum \Tc_{u_j}\Tc_{v_j}$ to have finite rank, $S=\sum_{l=1}^\rb(\cdot,g_l)f_l$ with some analytical functions $f_l,\ g_l$ is  $$\sum u_jv_j+w=\prod_{k=1}^\mb(1-|z_k|^2)^2\sum f_lg_l$$
and
$$w+\sum\overline{\Pb \overline{u_j}}\Pb v_j  {\mathrm is } \mb {  \mathrm harmonic}.$$
\end{theorem}
The proof of this result, as well as other ones in \cite{BRChoe3}, is based upon the finite rank theorems.

 \subsection{Landau  Hamiltonian and Landau-Toeplitz operators}  This    topic   was the source  of  the initial interest of  the author  in     Bergman-Toeplitz    operators.  About   the Landau  Hamiltonian one can find    a   detailed    information in  \cite{MelRoz},   \cite{RaiWar}, \cite{Bruneau},     and references therein. It  is  a   second  order   differential    operator    $H$ in $L_2(\C^1)=L_2(\R^2)$    that    describes   the dynamics    of  a   quantum particle    confined    to  a plane,  under   the action  of  the uniform magnetic    field     $B$  acting  orthogonal  to  the plane.  The operator   has  spectrum    consisting  of  eigenvalues $\L_q=(2q+1)B$, $q=0,1\dots,$ called    {\it Landau  levels},   with    corresponding   spectral    subspaces  $X_q$    having  infinite    dimension.  The subspace    $X_0$   is  closely related with    the Fock    space:  $X_0$   consists    of  the functions   $u(z)\in    L_2$,   $z\in\C^1$,   having  the form    $u(z)=\exp(-\frac{B|z^2|}{4})f(z)$,   where   $f(z)$  is  an  entire  analytical  function.

The other   spectral    subspaces   $X_q$   are obtained    from    $X_0$   by  the action  of  the so  called  \emph{creation    operator }   $\overline{Q}=(2i)^{-1}({\partial}+\frac{B}{2}(x_2-ix_1))$:
\begin{equation}\label{5.creation}
    X_q=\overline{Q}^qX_0.
\end{equation}

Under   the perturbation    by  the operator    of  multiplication  by  a   real    valued  function    $V(\xb)$,   tending to  zero    at  infinity,   the spectrum,   generally,   splits  into    clusters    around  Landau  levels, and a   lot of  interesting  results   were obtained,   describing   in  detail  the properties  of  these   clusters.  In  particular, for the case    of  the perturbation    $V$ having  compact support,    the infiniteness of  the clusters   has been    proved  only    under   the condition   that    $V$ has constant    sign    (or it's    minor  generalizations). It  remained    unclear whether it  is  possible    that    some    perturbation    would   not split   some    (or all)    clusters.

The methods of  the papers  cited   above   relate  this    question    with  the following   one:   is  it  possible    that    for some    function    $V$ with    compact support,    the \emph{Landau-Toeplitz}  operator    $\Tb_q(V)$  in  $X_q$  has finite  rank.   Here    $\Tb_q(V)$  is the operator $\Tb_q(V)u=P_qVu$    in $X_q$,   where   $P_q$   is  the orthogonal  projection  onto    $X_q$.  More exactly,   the eigenvalues of  $\Tb_q(V)$ give the main contribution to the spectrum of  $H+V$  near    $\L_q$. If   $\Tb_q(V)$ has finite  rank,   this    does    not immediately mean    that    the Landau  level  $\L_q$ does not split, but  only that it may split by the interaction with $L_{q'}$, $q'\ne q$. In such case we will say that there is no principal splitting.

The case    $q=0$   can be  treated directly    by  means   of  Luecking's  theorem.
The subspace $X_0$ consists of  analytical function multiplied by a Gaussian weight. Therefore, the infinite matrix  $\Ac(V)$    constructed by  means   of  the functions $f_j=\exp(-\frac{B|z^2|}{4})z^j$  is  the same    as  the matrix  $\Ac(F)$    for     the function $F=\exp(-\frac{B|z^2|}{2})V$, constructed by means of monomials $g_j=z^j$; these two matrices have a finite rank simultaneously. Thus, by Theorem \ref{3.Thm.Luecking}, the function    $F$,    and     consequently    the function $V$,   should  be  zero.   In  the initial terms,  this    means   that    the lowest  Landau  level   $\L_0$  necessarily principally splits  into    an  infinite    cluster, as soon as the perturbation $V$ is nonzero.

A   more    advanced technic is needed for higher Landau levels.    The subspaces  $X_q$ in    which   the  Toeplitz    operators   $\Tb_q(V)$ act,    do  not fit into    the framework   of      Luecking's  theorem.    We  can,    however,    use the relation    \eqref{5.creation} between different Landau subspaces.

The following   fact    has been    established in  \cite{Bruneau}, Corollary   9.3.
\begin{proposition}\label{5.brun} Let   $V$ be  a   bounded function with    compact support.  Then  for any $q$ the Toeplitz
operator    $\Tb_q(V)$  is  unitary equivalent  to  the operator    $\Tb_0(W)$, where $W=\Dc_q(\D)V$, $\Dc_q$   being  a   polynomial   of  degree      $q$ with    positive    coefficients.\end{proposition}
To  be  exact,  the statement   was proved  in  \cite{Bruneau}  for smooth  functions   $V$, but the proof extends  automatically   to  the case    of  bounded functions   $V$, and, of course, the expression $\Dc_q(\D)V$    should  be  understood  in  the sense   of  distribution.

Now,    suppose that    for some    $q$ the operator    $\Tb_q(V)$  has finite  rank.   By  Proposition \ref{5.brun},   the Toeplitz    operator    $\Tb_0(W)$  has finite  rank    as  well,   and we  can apply   Theorem \ref{3.Th.Alexroz}. So,    the distribution $W=\Dc_q(\D)V$  must        be  a   combination of  a   finite  number  of  the $\d$-distributions and their derivatives.  Therefore,   the Fourier transform   of  $W$ is  a   polynomial, the Fourier transform of $V$ must be a rational function, and such $V$ cannot be a function with compact support, by the analytic hypoellipticity.

We arrived at the following result.
\begin{theorem}\label{5:LandauLevels}Suppose    that    under   the perturbation    of  the Landau  Hamiltonian by a bounded function $V$ with compact support there    is  no  principal   splitting   for one of  Landau  levels.  Then    the perturbation    is zero, and therefore there is no principal splitting on other Landau levels either.\end{theorem}

\section{Other Bergman Spaces}

The analytical Bergman spaces, considered above, have a vast advantage, the multiplicative structure that is used all the time. For other types of Bergman spaces, without the multiplicative structure, the results are therefore less extensive. An exception is constituted by  the harmonic Bergman spaces in an even-dimensional space, due to their close relation to analytical functions.
\subsection{Harmonic Bergman spaces}
The aim of this section is to establish finite rank results for  Toeplitz operators in
 Bergman  spaces of harmonic functions. The results presented here  generalize the ones in \cite{AlexRoz}.

We start with the even-dimensional case, $d=2\mb$. Here the problem with harmonic spaces reduces easily to the
analytical Bergman spaces.

For a distribution $F\in \Ec'(\R^d)$ we consider a matrix
$H(F)$ consisting of elements $\langle F, f_j \overline{f_k}\rangle$, where $f_j$ is  some complete system of homogeneous
harmonic polynomials in $\R^{d}=\C^{\mb}$. It is convenient (but not obligatory) to suppose that  real and imaginary parts of analytic monomials, $\re(\zb^\a), \im(\zb^\a) ,$  $\a\in\Z_+^d,$  are among the polynomials $f_j$. For some subset $\Jc\subset \Z_+^d$, we denote by $H^{\Jc}(F)$ the matrix with entries $\langle F, f_j \overline{f_k}\rangle$, with $\re(\zb^\a), \im(\zb^\a), \ \a\in\Jc$ removed.
\begin{theorem}\label{ThHarmEven} Let $d=2\mb$ be an even integer.
Suppose that for some $N$, the set $\Jc$ satisfies the conditions of Theorem \ref{4.LeMulti}, and
for a distribution $F\in \Ec'(\R^n)$ the matrix $H^{\Jc}(F)$ has rank
$\rb\le N$. Then the distribution $F$ is a sum of $m\le \rb+1$ terms, each supported at one point:
$F=\sum L_q\d(\xb-\xb_q)$, $\xb_q\in \R^d$, $L_q$ are differential operators in $\R^d$.\end{theorem}
\begin{proof} We identify the space $\R^d$ with the complex space $\C^\mb$. Since the functions
$\zb^\a,\bar{\zb}^\b$ are harmonic, the matrix $\Ac^{\Jc}(F)$ (defined  in  Section 4) can be considered as a submatrix of $H^{\Jc}(F)$, and
therefore it has rank not greater than $\rb$. It remains to apply
Theorem \ref{4.LeMulti} to
establish that the distribution $F$ has the required form, with no more  than $\rb+1$ points $\xb_q$.
\end{proof}
The same reasoning establishes similar properties for the Bergman spaces of pluriharmonic and $\mb$-harmonic functions.

The odd-dimensional case requires considerably more work, and the results are less complete. We will use again  a kind of dimension
reduction, as in Theorem \ref{3.Th.Alexroz}, however, unlike the analytic case,  we will need projections of the distribution to
one-dimensional subspaces. We have to restrict our considerations to distributions being regular complex Borel measures (so we will use the notation $\m$ instead of $F$)  and from now on we will not consider the generalizations related with the removal of sparse subsets $\Jc$.

 Let $S$ denote the unit sphere in $\R^d$, $S=\{\z\in\R^d:|\z|=1\}$ and let $\s$ be the Lebesgue
measure on $S$. For $\z\in S$, we denote by $\Lc_\z$ the one-dimensional subspace in $\R^d$ passing
through $\z$, $\Lc_\z=\z \R^1$. For a measure with compact support $\m$ on $(\R^d)$ we define the measure
$\m_\z$ on $(\R^1)$ by setting $\langle \m_\z,\f\rangle=\langle \m,\f_z\rangle$, where $\f_z\in
C^\infty(\R^d)$ is $\f_z(\xb)=\f(\xb\cdot z)$. The measure $\m_\z$ can be understood as result of
projecting of $\m$ to $\Lc_\z$ with further transplantation of the projection, $\pi_*^{\Lc_\z}\m$, from
the line $\Lc_\z$ to the standard line $\R^1$.  The  Fourier transform $\Fc \m_\z$ of $\m_\z$ is
closely related with $\Fc \m$:
\begin{equation}\label{ProjDistr}
    \Fc (\m_\z)(t)=(\Fc \m)(t\z).
\end{equation}

The following   fact    in  the harmonic    analysis    of  measures    was established in  \cite{AlexRoz}.
\begin{proposition}\label{6:ARProp}For a finite complex Borel measure $\m$ with compact support in $\R^d$
the following three statements are equivalent:\\
a) $\m$ is discrete;\\
 b) $\m_\z$ is discrete for all $\z\in S$;\\
 c) $\m_\z$ is discrete for $\s$-almost all $\z\in S$.
\end{proposition}
 The    proof   of      Proposition \ref{6:ARProp}  can be  found   in  \cite{AlexRoz},  see   Corollary   5.3 there.

Now we return to our finite rank problem.
\begin{theorem}\label{TheoHarm}
Let $d\ge3$ be an odd integer, $d=2\mb+1$. Let $\m$ be a finite complex Borel measure in $\R^d$ with
compact support. Suppose that the matrix $H(\m)$ has finite rank $\rb$. Then $\supp \m$ consists of no
more than $\rb$ points.
\end{theorem}

\begin{proof}Fix some $\z\in S$ and choose some $d-1=2\mb$-dimensional linear subspace $\Lc\subset \R^{d}$ containing $\Lc_\z$.
 We choose  the co-ordinate system
$\xb=(x_1,\dots,x_d)$ in $\R^d$ so that the subspace $\Lc$ coincides with $\{\xb:x_d=0\}$. The
even-dimensional real space $\Lc$ can be considered as the $\mb$-dimensional complex space $\C^\mb$ with
co-ordinates $\zb=(z_1,\dots,z_\mb)$,  $z_j=x_{2j-1}+ix_{2j}$, $j=1,\dots,\mb$. The functions
$(\zb,x_d)\mapsto z^\a$, $(\zb,x_d)\mapsto \bar{\zb}^\b$, $\a,\b\in (\Z_+)^d$, are harmonic polynomials in
$\C^d\times  \R^1$. Moreover, by definition,
$\langle\m,\zb^\a\bar{\zb}^\b\rangle=\langle\pi^{\C^n}_*\m,\zb^\a\bar{\zb}^\b \rangle$. Hence,  the matrix
$\Ac(\pi^{\C^n}_*\m)$ is a submatrix of the matrix $H(\m)$, and the former has not greater rank
than the latter, $\rank(\Ac(\pi^{\C^n}_*\m))\le \rb$.   So we can apply Theorem \ref{ThHarmEven} and
obtain that the measure $\pi^{\C^\mb}_*\m$ is discrete and its support contains not more than $\rb$
points. Now we project the measure $\pi^{\C^n}_*\m$ to the real one-dimensional  linear subspace
$\Lc_\z$ in $\Lc$.  We obtain the same measure as if we had projected $\m$ to $\Lc_\z$ from the very
beginning, and not in two steps i.e.,  $\pi^{\Lc_\z}_*\m$. As a projection of a discrete measure,
$\pi^{\Lc_\z}_*\m$ is discrete and has no more than $\rb$ points in the support. By our definition of
the measure $\m_\z$ as $\pi^{\Lc_\z}_*\m$ transplanted to $\R^1$, this means  that $\m_\z$ is
discrete.

Due to the arbitrariness of the choice of $\z\in S$, we obtain that all measures $\m_\z$ are
discrete. Now we can apply Proposition \ref{6:ARProp} which implies that the measure $\m$ is discrete
itself. Finally, in order to show that the number of points in $\supp \m$ does not exceed $\rb$, we
chose $\z\in S$ such that no two points in $\supp \mu$ project to the same point in $\Lc_\z$. Then
the point masses of $\m$ cannot cancel each other under the projection, and thus $\card\supp
\m=\card\supp\m_\z\le \rb$.

The number of points in the support of $\m$ is estimated in the same way as in Theorem \ref{3.Th.Alexroz}.
\end{proof}

 The analysis of the reasoning in the proof shows that the only essential obstacle for extending Theorem
\ref{TheoHarm} to the case of distributions is the limitation set by Proposition \ref{6:ARProp}.
If we were able to prove this proposition for distributions, all other steps in the proof of Theorem
\ref{TheoHarm} would go through without essential changes.  However, it turns out that not only the
proof of Proposition \ref{6:ARProp} cannot be carried over to the distributional case, but,
moreover, the Corollary itself becomes wrong. The example, that can be found in \cite{AlexRoz},  does not disprove Theorem
\ref{TheoHarm} for distributions, however it indicates that the proof, if exists, should involve some
other ideas.

\subsection{Helmholtz Bergman spaces}\label{Helm.Sect} We consider now the Helmholtz equation
\begin{equation}\label{6:Helm.1}
    \Delta u +\kb^2 u=0,
\end{equation}
in $\O\subset\R^d$, where $\kb>0$ (we set $\kb^2=1$, without loosing in generality). Let $\O$ be a bounded domain and $F$ be a distribution with compact support in $\O$. We denote by $\Hs=\Hs(\O)$ the space of solutions of \eqref{6:Helm.1} in $\O$ belonging  to $L_2(\O)$ (we consider the Lebesgue measure here). We will call such solutions \emph{Helmholtz functions}. For a distribution $F\in \Ec'(\O)$ we, as usual, define the Toeplitz operator $\Tc_F:\Hs\to\Hs$, by means of the quadratic form
$(\Tc_f u, v)=\langle F, u\bar{v}\rangle$, $u,v\in \Hs$.
For any two systems of linearly independent functions $\Ss_1=\{f_j\},\Ss_2=\{g_k\}\subset \Hs$, we consider the matrix $\Ac=\Ac(F;\Ss_1,\Ss_2):$
\begin{equation}\label{6:matrix}
    \As(F;\Ss_1,\Ss_2)=(\langle F, f_j(x)\overline{g_k(x)}\rangle), \ f_j\in \Sigma_1, g_k\in\Sigma_2.
\end{equation}
If the Toeplitz operator $\Tc_F$ has finite rank $\rb$, the matrix \eqref{6:matrix} has rank not greater than $\rb$ for any $\Sigma_1,\Ss_2$. Moreover,   $$\rank \Tc_F= \max_{\Sigma_1,\Ss_2}\rank(\Ac(F;\Sigma_1,\Ss_2)).$$

\begin{theorem}\label{6:TheoremHelmholtz} Let $d\ge3$ and let $F$ be a function with compact support. Suppose that the Toeplitz operator $\Tc_F$   in  the space   of  Helmholtz   functions has finite rank. Then $F=0$.
\end{theorem}
\begin{proof} Consider the systems $\Ss_1, \Ss_2$ consisting of functions having the form $f_j(x)=e^{-ix_1}h_j(x'), \ g_k(x)=e^{ix_1}h_k(x')$, where $h_j(x')$ is an arbitrary system of harmonic functions of the variable $x'$ in the subspace $\Lc\subset \R^d:\ x_1=0$. Then the expression in \eqref{6:matrix} takes the form
\begin{equation}\label{6:ThHelm.1}
    \langle F, f_j(\xb)\overline{g_k(\xb)}\rangle=\int \int F(x_1,x')e^{-2ix_1}dx_1h_j(x')\overline{h_k(x')}dx'.
\end{equation}
This matrix has finite rank, not greater than $\rb$. Now we are in the conditions of Theorem \ref{TheoHarm} or Theorem \ref{ThHarmEven}, depending on whether $d$ is even or odd, in dimension $d-1\ge2$, applied to the function $\widetilde{F}(x')=\int F(x_1,x')e^{-2ix_1}dx_1$, i.e., the partial Fourier transform of $F$ in $x_1$ variable, calculated in the point $\x_1=2$. Since the matrix $\Ac(F;\Sigma_1,\Ss_2)$  has finite rank for arbitrary system of harmonic functions $h_j(x')$, by the above finite rank theorems about harmonic Bergman spaces, the function $\widetilde{F}(x')$ must be zero.
We make the Fourier transform of $\widetilde{F}$ in the remaining   variables    and obtain that the Fourier transform $\hat{F}(\x)$ of $F(\xb)$ equals zero for all $\x$ having the first component equal to $2$.

 Next we fix some $\o\in \R^d$, $|\o|=1$ and consider the system $\Ss_1=\Ss_2$ consisting of the functions having the form $f_j(\xb)=e^{-i\o \xb}h_j(x'),$ $g_k(x)=e^{i\o \xb}h_k(x')$ where $x'$ is the variable in the subspace $\Ls(\o)\subset\R^d$, orthogonal to $\o$, and $h_j$ are arbitrary harmonic functions. We repeat the reasoning above to obtain  that $\hat{F}(\x)=0$ for all $\x$ having the component in the direction of $\o$ equal to $2$. Now note that for any $\x\in\R^d,$ $|\x|\ge2,$ it is possible to find such $\o, \ |\o|=1$ that $\x$ has $\o$-component equal to $2$. Therefore we obtain that $\hat{F}(\x)=0$ for all $|\x|\ge2$.  So, we obtained that $\hat{F}$ has compact support. But, recall, $F$ also has compact support. Therefore $F$ must be zero.\end{proof}

 \subsection{An application: the Born approximation}\label{Born}
 In the quantum scattering theory one of the main objects to consider is the scattering matrix; details can be found in many books on the scattering theory, e.g., \cite{Yaf}, \cite{Rodb}.
 We consider the Born approximation, which (up to a constant factor) is the integral operator $\Kb$ with kernel
 \begin{equation}\label{6:Born.1}
    K(\o,\vs)=\int_{\R^d}F(\xb)e^{i\xb(\o-\vs)}d\xb,
 \end{equation}
 where $|\o|^2=|\vs|^2=E>0$, $F(\xb)$ is the potential (decaying at infinity sufficiently fast) and the operator acts on the sphere $S^{d-1}:$ $|\o|^2=E$. Further on, we suppose that $E=1$.

 The expression \eqref{6:Born.1} coincides with the quadratic form of the Toeplitz operator $\Tc_F$ in the space of solutions of the Helmholtz equation, considered on the  systems of   functions $e^{i\o \xb}, e^{i\vs \xb}$.
 We consider the case when the operator $\Kb$ has finite rank.  This implies that the matrix \eqref{6:matrix} has finite rank. So, applying Theorem \ref{6:TheoremHelmholtz}, we obtain the following result.
 \begin{theorem} Let $d\ge3$ and let $F$ be a function with compact support. Suppose that the Born approximation
  operator $\Kb$ has finite rank. Then $F=0$.\end{theorem}

 \end{document}